\documentclass[11pt]{article}
\usepackage{amsmath,amsthm,epsfig,amsfonts,color}
\allowdisplaybreaks

\def\singlespace{\def\baselinestretch{1.5}\@normalsize}

\newtheorem{theorem}{{Theorem}}
\newtheorem{lemma}{{Lemma}}
\newtheorem{remark}{{Remark}}
\renewcommand{\baselinestretch}{1.9}

\def\marginnote#1{\setbox0=\vtop{\hsize4pc
\small\raggedright\noindent\baselineskip9pt \rightskip=0.5pc plus
1.5pc #1}\leavevmode \vadjust{\dimen0=\dp0
\kern-\ht0\hbox{\kern-4.00pc\box0}\kern-\dimen0}}
\def\lboxit#1{\vbox{\hrule\hbox{\vrule\kern6pt
\vbox{\kern6pt#1\kern6pt}\kern6pt\vrule}\hrule}}

\begin{document}
\thispagestyle{empty}
\begin{center}
{\Large \textbf{Inference for  a Special  Bilinear Time Series  Model}}
\end{center}
\vskip 10 pt \centerline{\sc  Shiqing  Ling$^1$, Liang Peng$^2$ and Fukang Zhu$^3$}
\begin{footnotetext}
{\hspace*{-0.25 in}
$^1$Department of Mathematics, Hong Kong University of Science and Technology, Hong Kong, China.\\
$^2$Department of Risk Management and Insurance, Georgia State University, USA.\\
$^3$School of Mathematics, Jilin University, Changchun 130012, China.}
\end{footnotetext}

{\bf Abstract.}
It is well known that estimating bilinear models is quite challenging. Many different ideas have been proposed to solve this problem.
However, there is not a simple way to do inference even for its simple cases. This paper studies the special bilinear model
$$Y_t=\mu+\phi Y_{t-2}+ bY_{t-2}\varepsilon_{t-1}+ \varepsilon_t,$$
where $\{\varepsilon_t\}$ is a sequence of i.i.d. random variables with mean zero.
We first give a sufficient condition for the existence of a unique stationary solution for  the model and then propose a GARCH-type maximum likelihood estimator
for estimating the unknown parameters. It is shown that the GMLE is consistent and  asymptotically normal under only  finite fourth moment of errors.
Also a simple consistent estimator for the asymptotic covariance is provided.   A simulation study confirms the good finite sample performance.
Our estimation approach is novel and nonstandard and it may provide a new insight for future research in this direction.

\section{Introduction }
The general bilinear  time series model is defined by the equation
\begin{eqnarray}\label{eq1}
Y_t=\mu+\sum_{i=1}^p\phi_iY_{t-i}+\sum_{j=1}^q\psi_j\varepsilon_{t-j}+\sum_{l=1}^m\sum_{l'=0}^kb_{ll'} Y_{t-l}\varepsilon_{t-l'}+ \varepsilon_t,
\end{eqnarray}
where $\{\varepsilon_t\}$ is a sequence of independent and identically distributed random variables with mean zero and variance $\sigma^{2}$.
It was proposed by Granger and Anderson (1978a) and has been widely applied in many areas such as   control theory, economics and finance.
The structure of model (\ref{eq1})  has been studied in the literature especially for some special cases. For example, Subba Rao (1981) considered
model (\ref{eq1}) with $\psi_1=\cdots=\psi_q=0$; Davis and Resnick (1996) studied the asymptotic behavior  of the correlation function for the simple
bilinear model $Y_t=bY_{t-1}\varepsilon_{t-1}+\varepsilon_t$; Phan and Tran (1981), Turkman and Turkman (1997) and Basrak et al. (1999)
studied the model $Y_t=\phi_1Y_{t-1}+bY_{t-1}\varepsilon_{t-1}+\varepsilon_t$; Zhang and Tong (2001) considered the model
$Y_t=bY_{t-1}\varepsilon_t+\varepsilon_t$. A sufficient condition for stationarity of the general model was obtained by Liu and Brockwell (1988),
which is far away from the necessary one as pointed out by Liu (1989).
A simplified sufficient condition is given by Liu (1990a).

It is known that estimating the general bilinear model is quite challenging.  Many different ideas have been proposed to solve this problem for
some special cases of (\ref{eq1}), see Pham and Tran (1981), Guegan and Pham (1989), Wittwer (1989), Liu (1990b), Kim and Billard (1990), Kim et al. (1990),
Sesay and Subba Rao (1992), Gabr (1998) and Hili (2008). Extension to periodic bilinear models is studied by Bibi and Aknouche (2010) and Bibi and Gautier (2010).
 However, the asymptotic theory is either rarely established or only
derived by assuming that $\varepsilon_t$ follows a normal distribution in these papers. The Hellinger distance estimation in Hili (2008) even assumes that the density of $\varepsilon_t$ is known.
To understand this difficulty, let us look at the least squares estimator (LSE) considered by  Pham and Tran (1981). The LSE is equivalent
to the quasi-maximum likelihood estimator, which is the minimizer of
$$L_n(\theta)=\sum_{t=1}^n\varepsilon_t^2(\theta),$$
where $\theta$ is the vector  consisting of all parameters in the model and its true value is $\theta_0,\varepsilon_t(\theta_0)=\varepsilon_t$ and
$$\varepsilon_t(\theta)=Y_t-\mu-\sum_{i=1}^p\phi_iY_{t-i}-\sum_{j=1}^q\psi_j\varepsilon_{t-j}(\theta)
-\sum_{l=1}^m\sum_{l'=0}^kb_{ll'} Y_{t-l}\varepsilon_{t-l'}(\theta).$$
Given a sample $\{Y_1,\cdots, Y_n\}$, one needs  an efficient way to calculate the residual $\varepsilon_t(\theta)$ such that the effect from
the initial values  $\{Y_0, Y_{-1},\cdots\}$ is ignorable. This is the so-called invertibility of the model.
Although Liu (1990a) gave a sufficient condition for invertibility, it still remains unknown on how to use it to derive the asymptotic limit of the above LSE.
Another type of  invertibility was proposed by  Granger and Anderson (1978b).  That is, model \eqref{eq1} is said to be invertible if
$\lim\limits_{t\to\infty}E(\varepsilon_t-\hat{\varepsilon}_t)^2=0$,
where $\hat{\varepsilon}_{t}$ is an estimator of $\varepsilon_t$. Along this direction, the invertibility of a special bilinear model was studied
by Subba Rao (1981), Pham and Tran (1981) and Wittwer (1989). This type of invertibility may be useful for forecasting, but  it is not useful
for proving asymptotic normality of estimators of parameters.
This is because we need the property of $\varepsilon_t(\theta)$ at a neighborhood of the true parameter $\theta_0$ for deriving the asymptotic
limit of the estimator. For example, to obtain the asymptotic normality of the LSE, we  need the score function
$\dfrac{\partial\varepsilon_t(\theta)}{\partial\theta}$ to have a finite second moment, which in general results in some very restrictive
requirements for model (1). Let us further illustrate this issue as follows.

For the following simple bilinear model
\begin{eqnarray}\label{mod2}
Y_t=bY_{t-2}\varepsilon_{t-1}+\varepsilon_t,
\end{eqnarray}
 one needs $\prod\limits_{i=1}^mY_{t-i}$
has a finite moment for any $m$ in order to have $E\left\{\dfrac{\partial\varepsilon_t(\theta)}{\partial\theta}\right\}^2<\infty$.
Grahn (1995) showed that $EY_t^{2m}<\infty$ if and only if $b^{2m}E\varepsilon_t^{2m}<1$. Note that $E|Y_t|^m<\infty$ for any $m$ is equivalent to $b=0$
when $\varepsilon_t\sim N(0,\sigma^2)$. Thus, it is almost impossible to establish the
asymptotic normality of the LSE for model (\ref{mod2}) unless some special conditions are imposed. Instead Grahn (1995) proposed a nonstandard conditional LSE procedure for
model (\ref{mod2}) by using the facts that $E(Y_t^2|Y_s, s\le t-2)=\sigma^2+b^2\sigma^2Y_{t-2}^2$ and $E(Y_tY_{t-1}|Y_s, s\le t-2)=b\sigma^2Y_{t-2}$.
Although Grahn (1995) derived the asymptotic normality for the conditional LSE, the asymptotic variance and its estimator are not given, so some ad hoc method
such as bootstrap method is needed to construct confidence intervals for $b$.
Furthermore, the moment condition required is $EY_t^8<\infty$, which reduces to $b^8\sigma^8<1/105$ when $\varepsilon_t\sim N(0,\sigma^2)$.
This is quite restrictive on the parametric space of $(b,\sigma)$. When $\varepsilon_t\sim N(0,\sigma^2)$, Giordano (2000) and
Giordano and Vitale (2003) obtained the formula of the asymptotic variance for the conditional LSE of $b$, which can be estimated too.
Liu (1990b)  considered the LSE estimation for  the model
 \begin{eqnarray}\label{liu}
 Y_t=\phi Y_{t-p}+bY_{t-p}\varepsilon_{t-q}+\varepsilon_t,
 \end{eqnarray}
with $p\geq1$, and obtained its asymptotic normality by assuming that $\dfrac{\partial\varepsilon_t(\theta)}{\partial\theta}$ has a finite second moment.
As in model \eqref{mod2}, this condition may only hold when $b=0$ if $\varepsilon_t\sim N(0,\sigma^2)$. When $|\varepsilon_{t}|\le c$ (a constant) holds almost surely and $\phi=0$,
Liu (1990b) showed that this condition holds when $|b|\le\dfrac{1}{2c}$ which is a small parameter space when $c$ is large.
In general, one cannot check whether this condition holds when $\varepsilon_{t}$ is not bounded. That is, a general asymptotic theory for LSE or
maximum likelihood estimator (MLE) has not been established for  model \eqref{liu} up to now.

In this paper, we first give a sufficient condition for the existence of a unique stationary solution for a slightly more general model than (\ref{mod2}),
and then propose a GARCH-type MLE (GMLE) for estimating the unknown parameters. It is shown that the GMLE is consistent and asymptotically normal under only
finite fourth moment of errors. We organize this paper as follows. Section 2 presents our main results.
Section 3 reports some simulation results. Section 4 concludes. All proofs are given in Section 5.

\section{Estimation and Asymptotic Results}
Throughout we  consider the following special bilinear model:
\begin{equation}\label{mod3}
Y_t=\mu+\phi Y_{t-2}+ bY_{t-2}\varepsilon_{t-1}+ \varepsilon_t,
\end{equation}
where $\{\varepsilon_t\}$ is a sequence of independent and identically distributed random variables with mean zero and variance $\sigma^2>0$.  Let
$\ln^+x =\max\{\ln x,0\}$ be the positive part of the logarithm, and define
$$X_t=\begin{pmatrix}
Y_t\\
Y_{t-1}(\phi+b\varepsilon_t)
\end{pmatrix},~~~
A_t=\begin{pmatrix}
&0&1\\
&\phi+b\varepsilon_t&0\\
\end{pmatrix},~~~
B_t=\begin{pmatrix}
\mu+\varepsilon_t\\
0\end{pmatrix}.$$
Then (\ref{mod3}) can be rewritten as $X_t=A_tX_{t-1}+B_t$. It is easy to check that
\[\prod_{i=1}^{2m}A_i=\begin{pmatrix}
\prod\limits_{i=1}^m(\phi+b\varepsilon_{2i}) &0\\
0&\prod\limits_{i=1}^m(\phi+b\varepsilon_{2i-1})
\end{pmatrix},~~~
\prod_{i=1}^{2m+1}A_i=\begin{pmatrix}
0&\prod\limits_{i=1}^m(\phi+b\varepsilon_{2i})\\
\prod\limits_{i=1}^{m+1}(\phi+b\varepsilon_{2i-1})&0
\end{pmatrix}\]
for any integer $m\ge1$. For vector $x=(x_1,x_2)^\top$ and $2\times2$ matrix $y$, define $|x|=(x_1^2+x_2^2)^{1/2}$ and
$\|y\|=\max\limits_{|x|=1}|yx|$. Then
\[\ln\left\|\prod_{i=1}^{2m}A_i\right\|^2=\max\left\{\sum_{i=1}^m\ln(\phi+b\varepsilon_{2i})^2,
\sum_{i=1}^m\ln(\phi+b\varepsilon_{2i-1})^2\right\}\]
and
\[\ln\left\|\prod_{i=1}^{2m+1}A_i\right\|^2=\max\left\{\sum_{i=1}^m\ln(\phi+b\varepsilon_{2i})^2,
\sum_{i=1}^{m+1}\ln(\phi+b\varepsilon_{2i-1})^2\right\},\]
which imply that
\begin{equation*}\label{gamma}
\gamma=\lim_{n\to\infty}\frac 1n\ln\left\|\prod_{i=1}^nA_i\right\|=E\ln|\phi+b\varepsilon_1|.
\end{equation*}
Note that $E\ln^+|B_1|=E\ln^+|\mu+\varepsilon_1|$.
Therefore, when $E\ln^+|\mu+\varepsilon_1|<\infty$ and $E\ln|\phi+b\varepsilon_1|<0$,
it follows from Theorem 3.2.5 in Basrak (2000) that $X_n=B_n+\sum_{m=1}^{\infty}\prod_{i=0}^{m-1}A_{n-i}B_{n-m}$ converges almost surely
and is the unique strictly stationary solution of (\ref{mod3}).
Since we assume that $0<E\varepsilon_1^2<\infty$,  $E\ln^+|\mu+\varepsilon_1|<\infty$ holds naturally.

The following theorem summarizes the above arguments.
\begin{theorem}
Assume $E\ln|\phi+b\varepsilon_1|<0$. Then there exists a unique strictly stationary solution to model (\ref{mod3}), and
the solution is  ergodic and has the following representation:
$$Y_t=\mu+\varepsilon_t+\sum_{i=1}^\infty\prod_{r=0}^{i-1}(\phi+b\varepsilon_{t-2r-1})(\mu+\varepsilon_{t-2i}).$$
\end{theorem}

\begin{remark}
If the model (\ref{mod3}) is irreducible, then the condition $E\ln|\phi+b\varepsilon_1|<0$ is a necessary condition for stationarity,
which is a direct consequence of Bougerol and Picard (1992, Theorem 2.5).
From Theorem 3 in Kristensen (2009) we know that a sufficient condition for irreducibility is that $\varepsilon_t$ has a continuous component at zero and  $|\phi|<1$.
\end{remark}

\begin{remark}
It follows from Jensen's inequality that $2E\ln|\phi+b\varepsilon_1|=E\ln(\phi+b\varepsilon_1)^2
<\ln E(\phi+b\varepsilon_1)^2=\ln(\phi^2+\sigma^2b^2)$ for $b\neq 0$.  Hence  model (\ref{mod3}) is still stationary when $\phi^2+\sigma^2b^2=1$ and $b\neq 0$.
\end{remark}

\begin{remark}
When $P(\phi+b\varepsilon_1>0)=1$, results in Kesten (1973) can be employed to show that $Y_t$ has a heavy tail. However, it remains unknown
on the tail behavior of $Y_t$ when $P(\phi+b\varepsilon_1>0)<1$. This is in contrast to the well-studied simple bilinear model
$Y_t=\phi Y_{t-1}+b Y_{t-1}\varepsilon_{t-1}+\varepsilon_t$  in the literature, where the tail property has been clear,
but statistical inference for parameters remains unsolved when only some moment condition on $\varepsilon_t$ is assumed.
\end{remark}

Next we estimate the unknown parameters. Let $\mathcal{F}_t$ be the $\sigma$-fields generated by $\{\varepsilon_s: s\le t\}$.
Assume that $\{Y_1, Y_2,\cdots,Y_n\}$ are generated by model (\ref{mod3}). By noting that
\begin{align*}
E[Y_t|{\cal F}_{t-2}]&=\mu+\phi Y_{t-2},\\
{\rm Var}[Y_t|{\cal F}_{t-2}]&=E[(Y_t-\mu-\phi Y_{t-2})^2|{\cal F}_{t-2}]=\sigma^2(1+b^2Y_{t-2}^2),
\end{align*}
we propose to estimate parameters by maximizing the following quasi-log-likelihood function:
\begin{eqnarray*}
&&L_n(\theta)=\sum_{t=1}^n\ell_t(\theta)~~~\mbox{and}~~~
\ell_t(\theta)=-\frac{1}{2}\left[\ln[\sigma^2(1+b^2Y_{t-2}^2)]
+\frac{(Y_t-\mu-\phi Y_{t-2})^2}{\sigma^2(1+b^2Y_{t-2}^2)}\right],
\end{eqnarray*}
where $\theta=(\mu,\phi,\sigma^2,b^2)^\top$ is the unknown parameter and its true value is denoted by $\theta_0$.
The maximizer $\hat\theta_n$ of $L_n(\theta)$ is called the GMLE of $\theta_0$.
Although the estimation idea has appeared in
Francq and Zako\"{i}n (2004), Ling (2004) and Truquet and Yao (2012), the challenge is that $\left\{\dfrac{\partial\ell_t(\theta)}{\partial\theta}\right\}$ is no longer a martingale difference,
which complicates the derivation of the asymptotic limit.
A straightforward calculation shows that
\begin{align*}
&\frac{\partial\ell_t(\theta)}{\partial\mu}=\frac{Y_t-\mu-\phi Y_{t-2}}{\sigma^2(1+b^2Y_{t-2}^2)},\\
&\frac{\partial\ell_t(\theta)}{\partial\phi}=\frac{Y_{t-2}(Y_t-\mu-\phi Y_{t-2})}{\sigma^2(1+b^2Y_{t-2}^2)},\\
&\frac{\partial\ell_t(\theta)}{\partial\sigma^2}=-\frac{1}{2\sigma^2}\left[1-\frac{(Y_t-\mu-\phi Y_{t-2})^2}{\sigma^2(1+b^2Y_{t-2}^2)}\right],\\
&\frac{ \partial\ell_t(\theta)}{\partial b^2}=-\dfrac{Y_{t-2}^2}{2(1+b^2Y_{t-2}^2)}
\left[1-\frac{(Y_t-\mu-\phi Y_{t-2})^2}{\sigma^2(1+b^2Y_{t-2}^2)}\right].
\end{align*}
By solving
\[\sum_{t=1}^n\frac{\partial\ell_t(\theta)}{\partial\mu}=\sum_{t=1}^n\frac{\partial\ell_t(\theta)}{\partial\phi}
=\sum_{t=1}^n\frac{\partial\ell_t(\theta)}{\partial\sigma^2}=0,\]
we can write the GMLE for $\mu,\phi,\sigma^2$ explicitly in terms of $b^2$. Hence, using these explicit expressions and the equation
$\sum_{t=1}^n\dfrac{\partial\ell_t(\theta)}{\partial b^2}=0$, we can first obtain the GMLE for $b^2$, and then obtain the GMLE for $\mu,\phi,\sigma^2$.

It is easy to check that  $E\left[\dfrac{\partial \ell_t(\theta_0)}{\partial\theta}\Big|{\cal F}_{t-2}\right]=0$,
but $\left\{\dfrac{\partial \ell_t(\theta_0)}{\partial\theta}\right\}_{t=1}^{\infty}$ can not be a martingale difference. Therefore we can not use the
central limit theory for martingale difference to derive the asymptotic limit. Instead
we will show that $\left\{\dfrac{\partial \ell_t(\theta_0)}{\partial\theta}\right\}_{t=1}^{\infty}$ is a near-epoch dependent sequence so that the asymptotic
limit of the proposed GMLE can be derived. Denote
\begin{align*}
\Omega&=E\left[\frac{\partial\ell_t(\theta_0)}{\partial\theta}+\frac{\partial\ell_{t-1}(\theta_0)}{\partial\theta}\right]
\left[\frac{\partial\ell_t(\theta_0)}{\partial\theta}+\frac{\partial\ell_{t-1}(\theta_0)}{\partial\theta}\right]^\top-
E\left[\frac{\partial\ell_t(\theta_0)}{\partial\theta}\frac{\partial\ell_t(\theta_0)}{\partial\theta^\top}\right],\\
\Sigma&={\rm diag}\left\{
E\left[\frac{1}{\sigma^2_0(1+b_0^2Y_{t-2}^2)}\left(\begin{array}{cc}
1&Y_{t-2}\\
Y_{t-2}&Y_{t-2}^2
\end{array}\right)\right],\,\,
E\left(\begin{array}{cc}
\dfrac{1}{2\sigma^4_0}&\dfrac{Y_{t-2}^2}{2\sigma^2_0(1+b_0^2Y_{t-2}^2)}\\
\dfrac{Y_{t-2}^2}{2\sigma^2_0(1+b_0^2Y_{t-2}^2)}&\dfrac{Y_{t-2}^4}{2(1+b_0^2Y_{t-2}^2)^2}
\end{array}\right)
\right\}.
\end{align*}
The following theorem gives the asymptotic properties of the GMLE.

\begin{theorem}
Suppose the parameter space $\Theta$ is a compact subset of $\{\theta: E\ln|\phi+b\varepsilon_{1}|<0,
|\mu|\leq\bar\mu,|\phi|\leq\bar\phi,\underline\omega\leq\sigma^2\leq\overline\omega,\underline\alpha\leq b^2\leq\overline\alpha\}$, where
$\overline\mu,\overline\phi,\underline\omega,\overline\omega,\underline\alpha$ and $\overline\alpha$ are some finite positive constants,
and the true parameter value $\theta_0$ is an interior point in $\Theta$.
Further assume $E\varepsilon_1^4<\infty$.
Then as $n\to\infty$,

(a) $\hat\theta_n\to\theta_0$ almost surely,

(b) $\sqrt n(\hat{\theta}_n-\theta_0)\overset{d}{\to}N(0,\Sigma^{-1}\Omega\Sigma^{-1})$.
\end{theorem}

\begin{remark}
To ensure the positive definiteness of $\Sigma$  in Theorem 2,  we only need to show  the two sub-matrices are positive definite, which is equivalent to show the determinants of these two sub-matrices are positive.
Obviously  Cauchy-Schwarz inequality implies that the determinant of the second sub-matrix in $\Sigma$ is positive.
Put $A=1+b_0^2Y_{t}^2$, then the determinant of the first sub-matrix is
$\sigma_0^{-4}\{E(A^{-1})E(1-A^{-1})b_0^{-2} -[E(Y_tA^{-1})]^2\}
=\sigma_0^{-4}b_0^{-2}\{E(A^{-1})-[E(A^{-1})]^2-b_0^2[E(Y_tA^{-1})]^2\}
>\sigma_0^{-4}b_0^{-2}\{E(A^{-1})-E(A^{-2})-E(b_0^2Y_t^2A^{-2})\}=0.$
\end{remark}

\begin{figure}[h]
\centering
\includegraphics[scale=0.55]{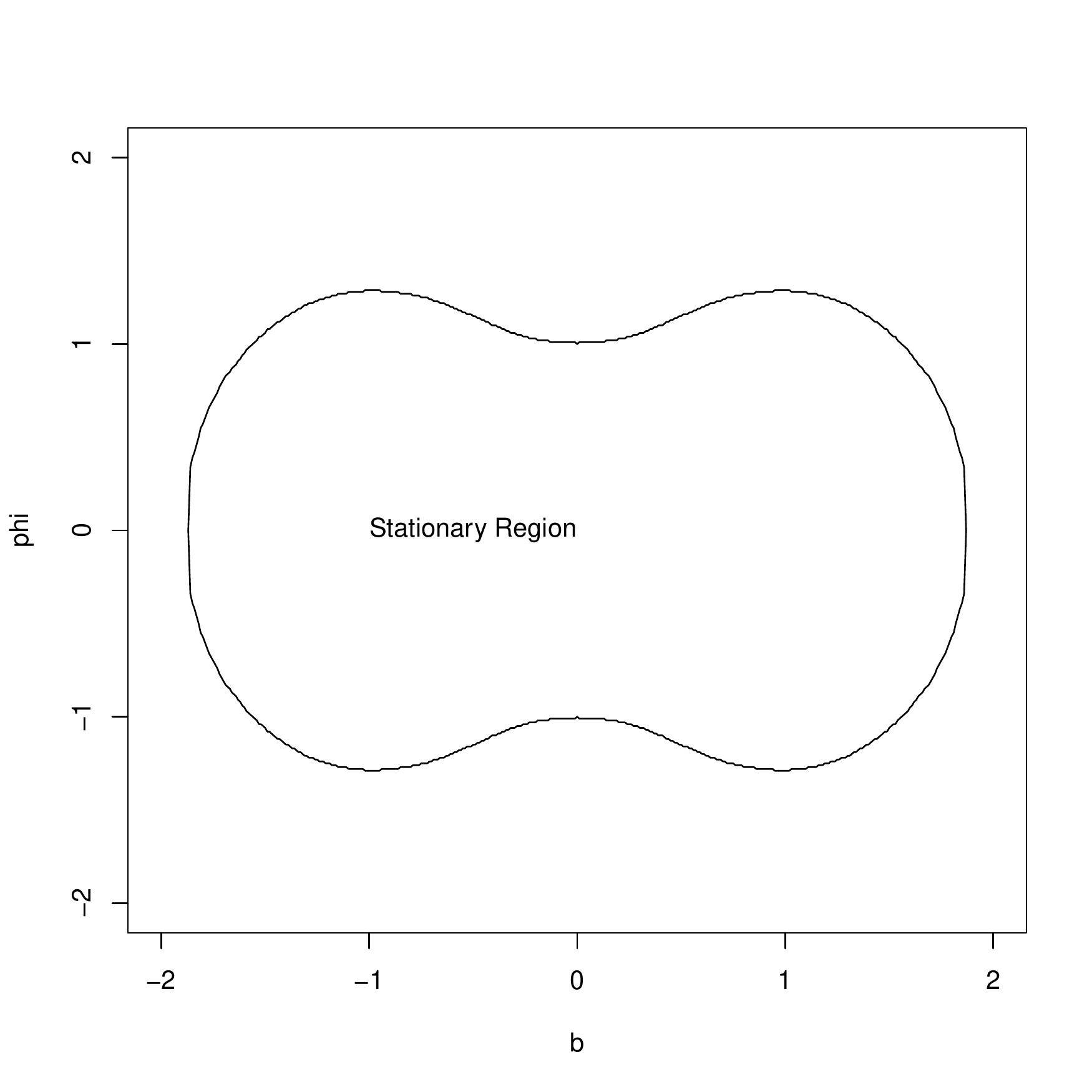}
\caption{Region of $(b,\phi)$ such that $E\ln|\phi+ b\varepsilon_1|<0$ when $\varepsilon_1\sim N(0,1)$.}

\end{figure}
\begin{remark}
Figure 1 gives the region of $(b, \phi)$ such that $E\ln|\phi+b\varepsilon_1|<0$ when $\varepsilon_1\sim N(0,1)$.
So $|b|$ can be greater than 1, i.e., the asymptotic limit of the proposed  GMLE holds under some weaker conditions than the
condition $|b\sigma|<105^{-1/8}\approx 0.5589$ required by the conditional LSE in Grahn(1995).
Moreover,
$\Omega$ and $\Sigma$ can be estimated simply by
\begin{eqnarray*}
\hat\Omega_n&=&\frac{1}{n}\sum_{t=1}^n\left[\frac{\partial l_t(\hat{\theta}_n)}
{\partial\theta}+\frac{\partial l_{t-1}(\hat{\theta}_n)}{\partial\theta}\right]
\left[\frac{\partial l_t(\hat{\theta}_n)}{\partial\theta}+\frac{\partial l_{t-1}(\hat{\theta}_n)}{\partial\theta}\right]^\top-
\frac{1}{n}\sum_{t=1}^n\left[\frac{\partial l_t(\hat{\theta}_n)}{\partial\theta}\frac{\partial l_t(\hat{\theta}_n)}{\partial\theta^\top}\right],\\
\hat{\Sigma}_{n}&=&{\rm diag}\left\{
\frac{1}{n}\sum_{t=1}^n\left[\frac{1}{\hat\theta_{n3}(1+\hat \theta_{n4}Y_{t-2}^2)}\left(\begin{array}{cc}
1&Y_{t-2}\\
Y_{t-2}&Y_{t-2}^2
\end{array}\right)\right]\right.,\\
&&~~~~~~~~\left.\frac{1}{n}\sum_{t=1}^{n}\left(\begin{array}{cc}
\dfrac{1}{2\hat\theta^2_{n3}}&\dfrac{Y_{t-2}^2}{2\hat\theta_{n3}(1+\hat\theta_{n4}Y_{t-2}^2)}\\
\dfrac{Y_{t-2}^2}{2\hat\theta_{n3}(1+\hat\theta_{n4}Y_{t-2}^2)}&\dfrac{Y_{t-2}^4}{2(1+\hat\theta_{n4}Y_{t-2}^2)^2}
\end{array}\right)\right\},
\end{eqnarray*}
respectively, where $\hat\theta_n=(\hat\theta_{n1},\hat\theta_{n2},\hat\theta_{n3},\hat\theta_{n4})^\top$.
The consistency follows from  Lemma 2 in Section 5.
\end{remark}

Since $\sum\limits_{t=1}^n\dfrac{\partial\ell_t(\theta)}{\partial b^2}=0$ is equivalent to
$\sum\limits_{t=1}^n\dfrac{\partial\ell_t(\theta)}{\partial b}=0$, one can not estimate $b$ by the above GMLE.
In order to estimate $b$, we need a consistent estimator for the sign of $b$. Write
\[ (Y_t-\mu-\phi Y_{t-2})(Y_{t-1}-\mu-\phi Y_{t-3})=\varepsilon_t(bY_{t-3}\varepsilon_{t-2}+\varepsilon_{t-1})
+b^2Y_{t-3}Y_{t-2}\varepsilon_{t-2}\varepsilon_{t-1}+bY_{t-2}\varepsilon_{t-1}^2 .\]
It is easy to see that
$E\{(Y_t-\mu-\phi Y_{t-2})(Y_{t-1}-\mu-\phi Y_{t-3})|{\cal F}_{t-2}\}=b\sigma^2Y_{t-2}$,
which motivates to estimate $b$ by minimizing the following least squares
\[\sum_{t=2}^n\{(Y_t-\mu-\phi Y_{t-2})(Y_{t-1}-\mu-\phi Y_{t-3})-b\sigma^2Y_{t-2}\}^2\]
with  $\mu, \phi $ and $\sigma^2$ being replaced by the corresponding GMLE. However,  in order to avoid requiring some moment conditions on $Y_t$, we propose to minimize the
weighted least squares
\[\sum_{t=2}^n\frac{\{(Y_t-\mu-\phi Y_{t-2})(Y_{t-1}-\mu-\phi Y_{t-3})-b\sigma^2Y_{t-2}\}^2}{(1+Y_{t-2}^2)\sqrt{1+Y_{t-3}^2}}\]
with $\mu, \phi, \sigma^2$ being replaced by the corresponding GMLE.
This results in
\[\tilde b_n=\left(\hat\theta_{n3}\sum_{t=2}^n\dfrac{Y_{t-2}^2}{(1+Y_{t-2}^2)\sqrt{1+Y_{t-3}^2}}\right)^{-1}
\sum_{t=2}^n\dfrac{(Y_t-\hat\theta_{n1}-\hat\theta_{n2}Y_{t-2})(Y_{t-1}-\hat\theta_{n1}-\hat\theta_{n2}Y_{t-3})Y_{t-2}}{
(1+Y_{t-2}^2)\sqrt{1+Y_{t-3}^2}}.\]
Like Theorem 2 (a), it is easy to show that $\tilde b_n=b+o_{p}(1)$.
Using $\tilde b_n$ to estimate the sign of $b$, we obtain an estimator for $b$ as $\hat b_n={\rm sgn}(\tilde b_n)\sqrt{\hat\theta_{n4}}$.
It easily follows from Theorem 2 that $\hat b_n=b+o_{p}(1)$ and
the asymptotic limit of $2b\sqrt n(\hat b_n-b)$ is the same as that of $\sqrt n(\hat\theta_{n4}-b^{2})$ given in Theorem 2. As stated in the simulation study, we propose to use $2b\sqrt n(\hat b_n-b)$ rather than $2\hat b_n\sqrt n(\hat b_n-b)$
to construct a confidence interval for $b$ although both share the same asymptotic limit.
Moreover we do not propose to estimate $b$ directly by $\tilde b_n$. The reason is that like Grahn (1995) we can not derive the formula and a
consistent estimator for the asymptotic variance of $\sqrt n(\tilde b_n-b)$. Moreover, $\tilde b_n$ is a less efficient estimator than $\hat b_n$ in general.

 Theorem 2 excludes the case of $b=0$, which reduces the bilinear model to a linear model. Hence testing $H_0: b=0$ is of interest.
Write $\Theta=[-\bar\mu, \bar\mu]\times [-\bar\phi, \bar \phi]\times [\underline\omega, \overline\omega]\times[0, \overline\alpha]$, where $\overline\mu,\overline\phi,\underline\omega,\overline\omega$ and $\overline\alpha$ are some finite positive constants. Then the case of $b=0$ means that
 $\theta=(\mu,\phi,\sigma^2,b^2)^\top$ lies at the boundary of the compact set $\Theta$, which implies that the case of $b=0$ is the well-known nonstandard situation of maximum likelihood estimation. The following theorem easily follows from Lemmas 1--3 in Section 5 and the same arguments in deriving (2.2) in Self and Liang (1987).

\begin{theorem}
Suppose the parameter space $\Theta$  satisfies $E\ln|\phi+b\varepsilon_1|<0$, and
 the true parameter value $\theta_0=(\mu_0,\phi_0,\sigma_0^2,0)^\top$ satisfies that $(\mu_0,\phi_0,\sigma_0^2)^\top$ is  an interior point of $[-\bar\mu, \bar\mu]\times [-\bar\phi, \bar \phi]\times [\underline\omega, \overline\omega]$.
Further assume $E\varepsilon_1^4<\infty$.
Then as $n\to\infty$,

(a) $\hat\theta_n\rightarrow\theta_0$ almost surely,

(b) $\sqrt n(\hat{\theta}_n-\theta_0)\overset{d}{\to}(Z_1,Z_2,Z_3,Z_4)^\top I(Z_4>0)
+(Z_1-\sigma_{14}\sigma_{44}^{-1}Z_4, Z_2-\sigma_{24}\sigma_{44}^{-1}Z_4,Z_3-\sigma_{34}\sigma_{44}^{-1}Z_4,0)^\top I(Z_4<0),$
where $(Z_1,Z_2,Z_3,Z_4)^\top\sim N(0, \Sigma^{-1}\Omega\Sigma^{-1})$, $\Sigma^{-1}\Omega\Sigma^{-1}=(\sigma_{ij})$ and $\Sigma$ and $\Omega$ are given in Theorem 2.
\end{theorem}

\begin{remark}
Using the consistent estimators for $\Omega$ and $\Sigma$ in Remark 5, one can easily simulate  the asymptotic limit of $\sqrt n(\hat\theta_n-\theta_0)$ so that interval estimation is obtained.
For testing $H_0: b=0$ against $H_a: b\neq 0$, we
let $\hat\sigma_{ij}$ denote the consistent estimator for $\sigma_{ij}$ given in Remark 5, but with $\hat\theta_{n,4}$ being replaced by 0.
By Theorem 3 one rejects $H_0$ at level $\xi$  whenever $\hat\theta_{n,4}>\sqrt{\hat\sigma_{44}}z_{\xi}/\sqrt n$, where $P(N(0,1)>z_{\xi})=\xi$.
We also remark that the likelihood ratio tests in Self and Liang (1987) do not apply to our bilinear model even for the case of $b^2>0$.
The reason is that $\Big\{\dfrac{\partial \ell_t(\theta)}{\partial\theta}\Big\}$ can not be a martingale difference,
and so $\Omega$ in Theorem 2 is different from the standard one
$E\Big\{\dfrac{\partial \ell_t(\theta)}{\partial\theta}\dfrac{\partial \ell_t(\theta)}{\partial\theta^\top}\Big\}$, which is necessary to ensure Wilks theorem holds for the likelihood ratio approach.
\end{remark}

\section{Simulation}
We investigate the finite sample performance of the proposed GMLE by drawing 1,000 random samples of size $n=200$ and 1,000 from model (\ref{mod3}) with
$\mu=0,b=\pm0.1$ or $\pm1, \phi=0$ or 0.9, and $\varepsilon_t\sim N(0,1)$. We compute the GMLE
$\hat\theta_n=(\hat\theta_{n1},\cdots,\hat\theta_{n4})^\top$ for $\theta=(\mu,\phi,\sigma^2, b^2)^\top$ and $\hat b_n$. For an estimator $\hat\beta$,
we use $E(\hat\beta)$, SD$(\hat\beta)$ and $\widehat{\rm SD}(\hat\beta)$ to denote the sample mean of $\hat\beta$, sample standard deviation of
$\hat\beta$ and sample mean of the standard deviation estimator given in Remark 5 of $\hat\beta$ based on the 1,000 samples.

Tables 1 and 2 report these quantities, which show that the proposed GMLE has a small bias (i.e, $E(\cdot)$ close to the true value) and
the proposed variance estimator is accurate too (i.e., $\widehat{\rm SD}(\cdot)$ close to SD$(\cdot)$). From these two tables, we  also observe that
SD$(\hat\beta)$ and $\widehat{\rm SD}(\hat\beta)$ are much smaller when $n=1000$ than those when $n=200$.
Although the proposed estimator
for $b$ has a small bias, the proposed variance estimator performs badly when $b$ is small. This is due to some very small values of $\hat\theta_{n4}$.
However, the variance estimator for $2b\hat b_n$ is reasonably well and much accurate than that for $\hat b_n$. Hence, we suggest to use $2b\sqrt n(\hat b_n -b)$ instead of
$\sqrt n(\hat b_n -b)$ to construct a confidence interval for $b$ in practice.

Next we use Remark 6 to test $H_0: b=0$ against $H_a: b\neq 0$ at levels 0.1 and 0.05. We draw 10,000 random samples of size $n=200$ and $1,000$ from model (\ref{mod3})
with $\mu=0, b=b^*/n^{0.25}, \phi=0.1$ or 0.9, $b^*=0, 0.5, 1$, and $\varepsilon_t\sim N(0,1)$.  The empirical size and power are reported in Table 3, where $b^*=0$ corresponds to the size.
From Table 3, we observe that the proposed test has a reasonably accurate size and non-trivial power.

\begin{table}
\label{T1}\caption{Sample mean and sample standard deviation are reported for the proposed GMLE for $(\mu,\phi,\sigma^2,b^2)^\top$
and $b$ with $n=200$.} \vspace{0.1in}
\begin{center}
\begin{tabular}{ccccccccc}
\hline
$(b, \phi)$& (0.1, 0) &(0.1, 0.9) &(-0.1, 0) &(-0.1, 0.9) &(1, 0) &(1, 0.9) &(-1, 0) &(-1, 0.9)\\
\hline
$E(\hat\theta_{n1})$&0.0008&-0.0020&0.0002&0.0021&-0.0014&0.0054&0.0039&-0.0009\\
SD$(\hat\theta_{n1})$&0.0707&0.0890&0.0720&0.0882&0.1058&0.1315&0.1015&0.1369\\
$\widehat{\rm SD}(\hat\theta_{n1})$&0.0701&0.0740&0.0703&0.0740&0.0999&0.1244&0.0997&0.1241\\
\hline
$E(\hat\theta_{n2})$&-0.0026&0.8793&-0.0092&0.8797&-0.0113&0.8887&-0.0052&0.8892\\
SD$(\hat\theta_{n2})$&0.0714&0.0399&0.0706&0.0400&0.1000&0.0878&0.0979&0.0891\\
$\widehat{\rm SD}(\hat\theta_{n2})$&0.0698&0.0352&0.0693&0.0350&0.0977&0.0868&0.0974&0.0869\\
\hline
$E(\hat\theta_{n3})$&0.9648&0.9896&0.9744&0.9919&1.0209&1.0254&1.0263&1.0280\\
SD$(\hat\theta_{n3})$&0.1115&0.1161&0.1047&0.1179&0.1976&0.2535&0.1993&0.2524\\
$\widehat{\rm SD}(\hat\theta_{n3})$&0.1196&0.1253&0.1209&0.1259&0.1855&0.2245&0.1880&0.2280\\
\hline
$E(\hat\theta_{n4})$&0.0355&0.0124&0.0344&0.0117&0.9987&1.0256&0.9882&1.0201\\
SD$(\hat\theta_{n4})$&0.0615&0.0166&0.0570&0.0168&0.3183&0.3205&0.3246&0.3041\\
$\widehat{\rm SD}(\hat\theta_{n4})$&0.0738&0.0174&0.0728&0.0172&0.2786&0.2667&0.2797&0.2671\\
\hline
$E(\hat b_n)$&0.0879&0.0780&-0.0986&-0.0736&0.9872&1.0015&-0.9812&-0.9996\\
SD$(\hat b_n)$&0.1666&0.0793&0.1572&0.0793&0.1555&0.1505&0.1593&0.1446\\
$\widehat{\rm SD}(\hat b_n)$&4.7444&0.8028&4.5939&0.8511&0.1379&0.1294&0.1393&0.1303\\
\hline
$E(2b\hat b_n)$&0.0176&0.0156&0.0197&0.0147&1.9744&2.0030&1.9624&1.9992\\
SD$(2b\hat b_n)$&0.0333&0.0159&0.0314&0.0159&0.3111&0.3010&0.3187&0.2892\\
$\widehat{\rm SD}(2b\hat b_n)$&0.0738&0.0174&0.0728&0.0172&0.2786&0.2667&0.2797&0.2671\\
\hline
\end{tabular}
\end{center}
\end{table}
\begin{table}
  \label{T2}\caption{Sample mean and sample standard deviation are reported for the proposed GMLE for $(\mu,\phi,\sigma^2,b^2)^\top$
  and $b$ with $n=1000$.} \vspace{0.1in}
\begin{center}
\begin{tabular}{ccccccccc}
\hline
$(b, \phi)$& (0.1, 0) &(0.1, 0.9) &(-0.1, 0) &(-0.1, 0.9) &(1, 0) &(1, 0.9) &(-1, 0) &(-1, 0.9)\\
\hline
$E(\hat\theta_{n1})$&0.0002&-0.0004&0.0001&0.0007&-0.0005&0.0031&0.0002&-0.0001\\
SD$(\hat\theta_{n1})$&0.0311&0.0330&0.0310&0.0331&0.0444&0.0563&0.0445&0.0556\\
$\widehat{\rm SD}(\hat\theta_{n1})$&0.0317&0.0325&0.0317&0.0325&0.0448&0.0549&0.0447&0.0549\\
\hline
$E(\hat\theta_{n2})$&-0.0044&0.8956&-0.0041&0.8958&-0.0036&0.8971&-0.0026&0.8972\\
SD$(\hat\theta_{n2})$&0.0306&0.0155&0.0311&0.0156&0.0428&0.0382&0.0434&0.0385\\
$\widehat{\rm SD}(\hat\theta_{n2})$&0.0317&0.0153&0.0376&0.0153&0.0444&0.0392&0.0444&0.0392\\
\hline
$E(\hat\theta_{n3})$&0.9907&1.0017&0.9904&1.0005&1.0037&1.0014&1.0024&1.0102\\
SD$(\hat\theta_{n3})$&0.0492&0.0555&0.0492&0.0555&0.0829&0.1033&0.0844&0.1042\\
$\widehat{\rm SD}(\hat\theta_{n3})$&0.0550&0.0564&0.0550&0.0563&0.0846&0.1019&0.0849&0.1029\\
\hline
$E(\hat\theta_{n4})$&0.0178&0.0095&0.0182&0.0097&0.9969&1.0050&0.9975&0.9949\\
SD$(\hat\theta_{n4})$&0.0247&0.0070&0.0251&0.0070&0.1365&0.1192&0.1358&0.1214\\
$\widehat{\rm SD}(\hat\theta_{n4})$&0.0327&0.0072&0.0329&0.0072&0.1284&0.1194&0.1298&0.1185\\
\hline
$E(\hat b_n)$&0.0935&0.0871&-0.0935&-0.0887&0.9961&1.0008&-0.9964&-0.9956\\
SD$(\hat b_n)$&0.0953&0.0434&0.0974&0.0432&0.0680&0.0589&0.0678&0.0606\\
$\widehat{\rm SD}(\hat b_n)$&1.8724&0.1336&1.9108&0.1349&0.0642&0.0594&0.0649&0.0593\\
\hline
$E(2b\hat b_n)$&0.0187&0.0174&0.0187&0.0177&1.9923&2.0015&1.9929&1.9912\\
SD$(2b\hat b_n)$&0.0191&0.0087&0.0195&0.0086&0.1359&0.1178&0.1356&0.1212\\
$\widehat{\rm SD}(2b\hat b_n)$&0.0327&0.0072&0.0329&0.0072&0.1284&0.1194&0.1298&0.1185\\
\hline
\end{tabular}
\end{center}
\end{table}

\begin{table}
  \label{T3}\caption{Empirical size and power are reported for the proposed test in Remark 6 for testing $H_0: b=0$ against $H_a: b\neq 0$ at levels 0.1 and 0.05.} \vspace{0.1in}
\begin{center}
\begin{tabular}{cccccccc}
\hline
&\multicolumn{3}{c}{level 0.1}&&\multicolumn{3}{c}{level 0.05}\\\cline{2-4}\cline{6-8}
$(n, \phi)$& $b=0$ &$b=0.5n^{-0.25}$ &$b=n^{-0.25}$ &&$b=0$ &$b=0.5n^{-0.25}$ &$b=n^{-0.25}$\\
\hline
(200, 0.1)&0.0790&0.1142&0.2513&&0.0449&0.0634&0.1550\\
\hline
(200, 0.9)&0.0777&0.3151&0.7575&&0.0422&0.2025&0.6267\\
\hline
(1000, 0.1)&0.0815&0.1237&0.3030&&0.0375&0.0592&0.1795\\
\hline
(1000, 0.9)&0.0758&0.4026&0.9604&&0.0336&0.2619&0.9052\\
\hline
\end{tabular}
\end{center}
\end{table}

\section{Conclusions}
Many different ideas have been proposed for estimating parameters in bilinear models. Unfortunately asymptotic limit is either missing or derived under some restrictive  distribution assumption on errors. By focusing  on a simple bilinear model,
we first give a sufficient condition for the existence of a unique stationary solution for the model and then propose a GARCH-type maximum likelihood estimator
for estimating parameters. It is shown that the proposed estimator is consistent and asymptotically normal under mild conditions. Moreover, the new estimation approach is novel, nonstandard and  has  good finite sample behavior.

\section{Proofs}
We first give one lemma, which plays a key role in the proofs of other lemmas.

\begin{lemma}
Under assumptions of Theorem 2,
\begin{eqnarray*}
(a)&&E\sup_{\theta\in\Theta}| \ell_t( \theta)|<\infty;\\
(b)&& E\ell_t(\theta)~\mbox{achieves its unique maximum value at}~\theta=\theta_{0}.
\end{eqnarray*}
\end{lemma}

\begin{proof} Since $E|\varepsilon_t|<\infty$, similar to the proof of Lemma 1 in Ling (2004), we can show that there exists a $\delta\in(0,1)$
such that $E|\phi+b\varepsilon_t|^\delta<1$. Using this and the expression of $Y_t$ in Theorem 1, we can show that  $E|Y_t|^\delta<\infty$.
Take $\delta_0=\delta/2$. Thus, by Jensen's inequality, it follows that
\begin{align*}
E\sup_{\theta\in\Theta}|\ln[\sigma^2(1+b^{2}Y_{t-2}^2)]|
&\le\sup_{\theta\in\Theta}|\ln\sigma^2|+\delta^{-1}_0E\ln(1+\overline\alpha Y_{t-2}^2)^{\delta_0}\\
&\le|\ln\overline\omega|+\delta^{-1}_0\ln(1+\overline\alpha^{\delta_0}E|Y_{t-2}|^\delta)<\infty,
\end{align*}
where the following elementary relationship is used: $(a_1+a_2)^s\le a_1^s+a_2^s$ for all $a_1,a_2>0$ and $s\in[0,1]$.
Furthermore, since $Y_t-\mu-\phi Y_{t-2}=\varepsilon_t-(\mu-\mu_0)-(\phi-\phi_0)Y_{t-2}+b_0\varepsilon_{t-1}Y_{t-2}$, it can be shown that
\begin{align}\label{4.1}
&E\sup_{\theta\in\Theta}\frac{(Y_t-\mu-\phi Y_{t-2})^2}{\sigma^2(1+b^2Y_{t-2}^2)}\nonumber\\
&~~~\le\frac{4}{\underline\omega}\left[E\sup_{\theta\in\Theta}\frac{\varepsilon_t^2}{1+b^2Y_{t-2}^2}
+E\sup_{\theta\in\Theta}\frac{(\mu-\mu_0)^2}{1+b^2Y_{t-2}^2}+E\sup_{\theta\in\Theta}\frac{(\phi-\phi_{0})^2Y_{t-2}^2}{1+b^2Y_{t-2}^2}
+E\sup_{\theta\in\Theta}\frac{b_0^2\varepsilon_{t-1}^2Y_{t-2}^2}{1+b^2Y_{t-2}^2}\right]\nonumber\\
&~~~\le\frac{4}{\underline\omega}\left(\overline\omega+4\overline\mu^2+\frac{4\overline\phi^2}{\underline\alpha}
+\frac{\overline\omega~\overline\alpha}{\underline\alpha}\right)
<\infty.
\end{align}
Hence, (a) holds.

To prove (b), by noting that
\begin{align*}
E[(Y_t-\mu-\phi Y_{t-2})^2|{\cal F}_{t-2}]
&=E[(\varepsilon_t-(\mu-\mu_0)-(\phi-\phi_{0})Y_{t-2}+b_0\varepsilon_{t-1}Y_{t-2})^{2}|{\cal F}_{t-2}]\\
&=[(\mu-\mu_0)+(\phi-\phi_0)Y_{t-2}]^2+\sigma^2_0(1+b_0^2Y_{t-2}^2),
\end{align*}
we have
\begin{align}\label{4.2}
E\ell_t(\theta)&=-\frac{1}{2}E\left[\ln[\sigma^2(1+b^2Y_{t-2}^2)]+\frac{(Y_{t}-\mu-\phi Y_{t-2})^2}{\sigma^2(1+b^2Y_{t-2}^2)}\right]\nonumber\\
&=-\frac{1}{2}\left\{E\ln[\sigma^2(1+b^2Y_{t-2}^2)]+E\frac{\sigma^2_0(1+b_0^2Y_{t-2}^2)}{\sigma^2(1+b^2Y_{t-2}^2)}\right\}
-E\frac{[(\mu-\mu_0)+(\phi-\phi_0)Y_{t-2}]^2}{2\sigma^2(1+b^2Y_{t-2}^2)}.
\end{align}
The second term in \eqref{4.2} reaches its maximum at zero, and  this occurs if and only if $\mu=\mu_0$ and $\phi=\phi_0$.
The first term in \eqref{4.2} is equal to
\begin{equation}\label{4.3}
-\dfrac{1}{2}[-E(\ln M_t)+EM_t]-\dfrac{1}{2}E\ln[\sigma^2_0(1+b_0^2Y_{t-2}^2)],
\end{equation}
where $M_t=\dfrac{\sigma^2_0(1+b_0^2Y_{t-2}^2)}{\sigma^2(1+b^2Y_{t-2}^2)}$. As in Lemma 2 of Ling (2004), \eqref{4.3} reaches its maximum
$-1/2-E\ln[\sigma^2_0(1+b_0^2Y_{t-2}^2)]/2$, and this occurs if and only if $\sigma^2=\sigma_0^2$ and $b^2=b_0^2$.
Thus, $E\ell_t(\theta)$ is uniquely maximized at $\theta_0$.  \end{proof}

\begin{lemma}
Under assumptions of Theorem 2,
\begin{eqnarray*}
(a)&& E\sup_{\theta\in\Theta}\left\|\frac{\partial\ell_t(\theta)}{\partial\theta}\right\|^2<\infty,\\
(b)&& E\sup_{\theta\in \Theta}\left[\frac{\partial^2\ell_t(\theta)}{\partial \theta\partial\theta^\top}\right]<\infty,\\
(c)&&\sup_{\theta\in\Theta}\left|\frac{1}{n}\sum_{t=1}^n\left\{\frac{\partial^2\ell_t(\theta)}{\partial \theta\partial\theta^\top}
-E\frac{\partial^2\ell_t(\theta)}{\partial \theta\partial\theta^\top}\right\}\right|=o_p(1)\quad\text{as}\quad n\to\infty,\\
(d)&&\sup_{\theta\in\Theta}\left|\frac 1{n}\sum_{t=1}^n\left\{\left\|\frac{\partial\ell_t(\theta)}{\partial\theta}\right\|^2-E\left\|\frac{\partial\ell_t(\theta)}{\partial\theta}\right\|^2\right\}\right|=o_p(1)\quad\text{as}\quad n\to\infty.
\end{eqnarray*}
\end{lemma}

\begin{proof}
As in \eqref{4.1},  it is easy to show that
\begin{align*}
E\sup_{\theta\in\Theta}\left[\frac{\partial\ell_t(\theta)}{\partial\phi}\right]^2&=
E\sup_{\theta\in\Theta}\frac{Y_{t-2}^2[\varepsilon_{t}-(\mu-\mu_0)
-(\phi-\phi_0)Y_{t-2}+b_0\varepsilon_{t-1}Y_{t-2}]^2}{\sigma^4(1+b^2Y_{t-2}^2)^2}\\
&\le\frac{4}{\underline\omega^2}\left[E\sup_{\theta\in\Theta}\frac{\varepsilon_t^2Y_{t-2}^2}{(1+b^2Y_{t-2}^2)^2}
+E\sup_{\theta\in\Theta}\frac{(\mu-\mu_0)^2Y_{t-2}^2}{(1+b^2Y_{t-2}^2)^2}\right.\\
&~~~\left.+E\sup_{\theta\in\Theta}\frac{(\phi-\phi_{0})^2Y_{t-2}^4}{(1+b^2Y_{t-2}^2)^2}
+E\sup_{\theta\in\Theta}\frac{b_0^2\varepsilon_{t-1}^2Y_{t-2}^4}{(1+b^2Y_{t-2}^2)^2}\right]<\infty,\\
E\sup_{\theta\in\Theta}\left[\frac{\partial\ell_t(\theta)}{\partial b^{2}}\right]^2&
\le\frac{1}{2}E\sup_{\theta\in\Theta}\frac{Y_{t-2}^4}{(1+b^2Y_{t-2}^2)^2}+\frac{1}{2}E\sup_{\theta\in\Theta}
\frac{Y_{t-2}^4(Y_t-\mu-\phi Y_{t-2})^4}{\sigma^4(1+b^2Y_{t-2}^2)^4}\\
&=\frac{1}{2}E\sup_{\theta\in\Theta}\frac{Y_{t-2}^4}{(1+b^2Y_{t-2}^2)^2}\\
&~~~+\frac{1}{2}E\sup_{\theta\in\Theta}\frac{Y_{t-2}^4[\varepsilon_{t}-(\mu-\mu_0)
-(\phi-\phi_0)Y_{t-2}+b_0\varepsilon_{t-1}Y_{t-2}]^4}{\sigma^4(1+b^2Y_{t-2}^2)^4}\\
&\le\frac{1}{2\underline\alpha^2}+
2E\sup_{\theta\in\Theta}\frac{Y_{t-2}^4[\varepsilon_{t}^2+(\mu-\mu_0)^2
+(\phi-\phi_0)^2Y_{t-2}^2+b_0^2\varepsilon_{t-1}^2Y_{t-2}^2]^2}{\sigma^4(1+b^2Y_{t-2}^2)^4}\\
&\le\frac{1}{2\underline\alpha^2}+\frac{8}{\underline\omega^2}
\left[E\sup_{\theta\in\Theta}\frac{\varepsilon_t^4Y_{t-2}^4}{(1+b^2Y_{t-2}^2)^4}
+E\sup_{\theta\in\Theta}\frac{(\mu-\mu_0)^4Y_{t-2}^4}{(1+b^2Y_{t-2}^2)^4}\right.\\
&~~~\left.+E\sup_{\theta\in\Theta}\frac{(\phi-\phi_{0})^4Y_{t-2}^8}{(1+b^2Y_{t-2}^2)^4}
+E\sup_{\theta\in\Theta}\frac{b_0^4\varepsilon_{t-1}^4Y_{t-2}^8}{(1+b^2Y_{t-2}^2)^4}\right]<\infty.
\end{align*}
Similarly, we can show that  other  terms   in (a) are finite too. Hence, (a) holds.

A straightforward calculation gives that
\begin{align*}
\frac{\partial^2\ell_t(\theta)}{\partial\mu^2}&=-\frac{1}{\sigma^2(1+b^2Y_{t-2}^2)},\\
\frac{\partial^2\ell_t(\theta)}{\partial\phi^2}&=-\frac{ Y_{t-2}^2}{\sigma^2(1+b^2Y_{t-2}^2)},\\
\frac{\partial^2\ell_t(\theta)}{\partial\sigma^4}&=\frac{1}{2\sigma^4}\left[1-\frac{2(Y_t-\mu-\phi Y_{t-2})^2}{\sigma^2(1+b^2Y_{t-2}^2)}\right],\\
\frac{\partial^2\ell_t(\theta)}{\partial b^4}&=
\dfrac{Y_{t-2}^4}{2(1+b^2Y_{t-2}^2)^2}\left[1-\frac{2(Y_t-\mu-\phi Y_{t-2})^2}{\sigma^2(1+b^2Y_{t-2}^2)}\right],\\
\frac{\partial^2\ell_t(\theta)}{\partial\mu\partial\phi}&=-\frac{Y_{t-2}}{\sigma^2(1+b^2Y_{t-2}^2)},\\
\frac{\partial^2\ell_t(\theta)}{\partial\mu\partial\sigma^2}&=-\frac{Y_t-\mu-\phi Y_{t-2}}{\sigma^4(1+b^2Y_{t-2}^2)},\\
\frac{\partial^2\ell_t(\theta)}{\partial\mu\partial b^2}&=-\frac{Y_{t-2}^2(Y_t-\mu-\phi Y_{t-2})}{\sigma^2(1+b^2Y_{t-2}^2)^2},\\
\frac{\partial^2\ell_t(\theta)}{\partial\phi\partial\sigma^2}&=-\frac{Y_{t-2}(Y_t-\mu-\phi Y_{t-2})}{\sigma^4(1+b^2Y_{t-2}^2)},\\
\frac{\partial^2\ell_t(\theta)}{\partial\phi\partial b^2}&=-\frac{Y_{t-2}^3(Y_t-\mu-\phi Y_{t-2})}{\sigma^2(1+b^2Y_{t-2}^2)^2},\\
\frac{\partial^2\ell_t(\theta)}{\partial\sigma^2\partial b^2}&=-\frac{Y_{t-2}^2(Y_t-\mu-\phi Y_{t-2})^2}{2\sigma^4(1+b^2Y_{t-2}^2)^2}.
\end{align*}
Using these formulas and some   similar  arguments  in proving (a), we can show that (b) holds.  (c) and (d) follow from  Theorem 3.1 in Ling and McAleer (2003). \end{proof}

\begin{lemma}
Under assumptions of Theorem 2,
$$\frac{1}{\sqrt n}\sum_{t=1}^n\frac{\partial\ell_t(\theta_0)}{\partial\theta}\overset{d}{\to}N(0,\Omega).$$
\end{lemma}

\begin{proof}  Let $C=(c_1, c_2, c_3, c_4)^\top$ be any constant vector with $C^\top C\not=0$ and define
\begin{eqnarray*}
S_n&\equiv&\frac{C^\top}{\sqrt n}\sum_{t=1}^n\frac{\partial l_t(\theta_0)}{\partial\theta}\\
&=&\frac{1}{\sqrt n}\sum_{t=1}^n\left[
\frac{c_1\xi_{1t}}{\sqrt{\sigma^2_0(1+b_0^2Y_{t-2}^2)}}+\frac{c_2Y_{t-2}\xi_{1t}}{\sqrt{\sigma^2_0(1+b_0^2Y_{t-2}^2)}}-
\frac{c_3\xi_{2t}}{2\sigma^2_0}-\frac{c_4Y_{t-2}^{2}\xi_{2t}}{2(1+b_0^2Y_{t-2}^2)}\right]\\
&\equiv&\frac{1}{\sqrt n}\sum_{t=1}^ns_t,
\end{eqnarray*}
where
\[\xi_{1t}=\frac{\varepsilon_t+b_0Y_{t-2}\varepsilon_{t-1}}{\sqrt{\sigma^2_0(1+b_0^2Y_{t-2}^2)}},~~~\xi_{2t}=1-
\frac{(\varepsilon_t+b_0Y_{t-2}\varepsilon_{t-1})^2}{\sigma^2_0(1+b_0^2Y_{t-2}^2)}.\]
Since $E(s_ts_{t+k})=0$ if $|k|\ge2$, we have
$$\sigma_n^2\equiv E\left(\frac{1}{\sqrt n}\sum_{t=1}^ns_t\right)^2\to Es_t^2+2Es_ts_{t-1}=C^\top\Omega C,$$
as $n\to\infty$. Next we show that $\{s_t\}$ is $L^2(\nu)$-near-epoch dependent series, that is,
\begin{equation}\label{4.4}
E[s_t-E(s_t|{\cal F}_t^m)]^2=O(m^{-\nu}),
\end{equation}
for any $\nu>2$ and large $m$, where ${\cal F}_t^m=\sigma\{\varepsilon_t,\cdots,\varepsilon_{t-m}\}$. Put
$$Y_{m,t}=\mu_0+\varepsilon_t+\sum_{1\le i\le m/2-1}\prod_{r=0}^{i-1}(\phi_0+b_0\varepsilon_{t-2r-1})(\mu_0+\varepsilon_{t-2i}).$$
Then $Y_{m,t}\in{\cal F}_t^m$.  From the proof of Lemma 1, there exists a $\delta\in(0,1)$
such that $E|\phi_0+b_0\varepsilon_t|^\delta<1$. Thus, by Theorem 1, we have
\begin{align}
E|Y_t-Y_{m,t}|^{\delta}&=E\left|\sum_{i\ge m/2}\prod_{r=0}^{i-1}(\phi_0+b_0\varepsilon_{t-2r-1})(\mu_0+\varepsilon_{t-2i})\right|^{\delta}\nonumber\\
&\le\sum_{i\ge m/2}E\left|\prod_{r=0}^{i-1}(\phi_0+b_0\varepsilon_{t-2r-1})(\mu_0+\varepsilon_{t-2i})\right|^{\delta}\nonumber\\
&=\sum_{i\ge m/2}\prod_{r=0}^{i-1}E|\phi_0+b_0\varepsilon_{t-2r-1}|^{\delta}E|\mu_0+\varepsilon_{t-2i}|^{\delta}\nonumber\\
&=O\left(\sum_{i\ge m/2}(E|\phi_0+b_0\epsilon_1|^{\delta})^{i-1}\right)\nonumber\\
&=O(\rho^{m}),\label{4.5}
\end{align}
 where $\rho\in (0,1)$.  It follows from (\ref{4.5}) that
\begin{align}
&E\left|\frac{Y_{t-2}}{{1+b_{0}^{2} Y_{t-2}^{2}}}-\frac{Y_{m,t-2}}{{1+b_{0}^{2} Y_{m,t-2}^{2}}}\right|\nonumber\\
&~~~=\frac{2}{|b_0|}E\left|\frac{b_0Y_{t-2}}{{2(1+b_{0}^{2} Y_{t-2}^{2})}}-\frac{b_0Y_{m,t-2}}{{2(1+b_{0}^{2} Y_{m,t-2}^{2})}}\right|\nonumber\\
&~~~\le\frac{2}{|b_0|}E\left|\frac{b_0Y_{t-2}}{{2(1+b_{0}^{2} Y_{t-2}^{2})}}-\frac{b_0Y_{m,t-2}}{{2(1+b_{0}^{2} Y_{m,t-2}^{2})}}\right|^{\delta}\nonumber\\
&~~~\le\frac2{|b_0|}\left\{E\left|\frac{b_0Y_{t-2}}{2(1+b_0^2Y_{t-2}^2)}-\frac{b_0Y_{m,t-2}}{2(1+b_0^2Y_{t-2}^{2})}\right|^{\delta}
+E\left|\frac{b_0Y_{m,t-2}}{2(1+b_0^2Y_{t-2}^2)}-\frac{b_0Y_{m,t-2}}{2(1+b_0^2Y_{m,t-2}^2)}\right|^{\delta}\right\}\nonumber\\
&~~~\le\frac{4|b_0|^{\delta}}{|b_0|}E|Y_{t-2}-Y_{m,t-2}|^{\delta}\nonumber\\
&~~~=O(\rho^{m}),\label{add1}
\end{align}
 which implies that
\begin{align}
&E\left|\frac{Y_{t-2}}{{1+b_0^2Y_{t-2}^2}}-E\left(\frac{Y_{t-2}}{{1+b_0^2Y_{t-2}^2}}\Big|{\cal F}_{t}^{m}\right)\right|\nonumber\\
&~~~\le E\left|\frac{Y_{t-2}}{{1+b_0^2 Y_{t-2}^2}}-\frac{Y_{m,t-2}}{{1+b_0^2Y_{m,t-2}^2}}\right|
+E\left[E\left(\left|\frac{Y_{t-2}}{{1+b_{0}^2Y_{t-2}^2}}-\frac{Y_{m,t-2}}{{1+b_0^2Y_{m,t-2}^{2}}}\right|\Big|{\cal F}_{t}^{m}\right)\right]\nonumber\\
&~~~=2E\left|\frac{Y_{t-2}}{{1+b_0^2Y_{t-2}^2}}-\frac{Y_{m,t-2}}{{1+b_0^2Y_{m,t-2}^2}}\right|\nonumber\\
&~~~=O(\rho^{m}).\label{add2}
\end{align}
Similar to (10), we can show hat
\begin{equation}\label{add3}
E\left|\frac{Y_{t-2}^{2}}{{1+b_{0}^{2} Y_{t-2}^{2}}}-E\left(\frac{Y_{t-2}^{2}}{{1+b_{0}^{2} Y_{t-2}^{2}}}\Big|{\cal F}_{t}^{m}\right)\right|=O(\rho^{m})
\end{equation}
for some $\delta\in(0,1)$.  Furthermore,  since $\varepsilon_{t}$ and $\varepsilon_{t-1}$  are independent of $Y_{t-2}$ and
$Y_{t-2}^{2}/(1+b_{0}^{2} Y_{t-2}^{2})$ is bounded, it follows  from
(\ref{add2}) and (\ref{add3}) that
\begin{align*}
&E\left|\frac{Y_{t-2}\xi_{1t}}{\sqrt{1+b_0^2Y_{t-2}^2}}-E\left(\frac{Y_{t-2}\xi_{1t}}{\sqrt{1+b_0^2Y_{t-2}^2}}\Big|{\cal F}_t^m\right)\right|^2\\
&\le2E\varepsilon_{t}^2E\left|\frac{Y_{t-2}}{1+b_0^2Y_{t-2}^2}-E\left(\frac{Y_{t-2}}{1+b_0^2Y_{t-2}^2}\Big|{\cal F}_t^m\right)\right|^2
+2b_0^2E\varepsilon_{t-1}^2E\left|\frac{Y_{t-2}^2}{1+b_0^2Y_{t-2}^2}-E\left(\frac{Y_{t-2}^2}{1+b_0^2 Y_{t-2}^2}\Big|{\cal F}_t^m\right)\right|^2\\
&=O\left(E\left|\frac{Y_{t-2}}{1+b_0^2Y_{t-2}^2}-E\left(\frac{Y_{t-2}}{1+b_0^2Y_{t-2}^2}\Big|{\cal F}_t^m\right)\right|\right)
+O\left(E\left|\frac{Y_{t-2}^2}{1+b_0^2Y_{t-2}^2}-E\left(\frac{Y_{t-2}^2}{1+b_0^2Y_{t-2}^2}\Big|{\cal F}_t^m\right)\right|\right)\\
&=O(\rho^m).
\end{align*}
Similar inequalities hold for other terms in $s_{t}$ and hence \eqref{4.4} holds. Therefore we conclude that $S_{n}\overset{d}{\to} N(0,C^\top\Omega C)$ by
Theorem 21.1 in Billingsley (1968). Furthermore, by
the Cram\'{e}r-Wold device, we complete the proof. \end{proof}

\begin{proof}[Proof of  Theorem 2] Part (a) follows  from Theorem 1(a) in Ling and McAleer (2010) and Lemma 1 (Assumption 2(i) in that paper
 automatically holds since we only need one initial value).  First, part (a) of this theorem implies that $\hat{\theta}_n$
 converges a.s.  to $\theta_0$.  Second,
$\dfrac{1}{n}\sum\limits_{t=1}^n\dfrac{\partial^2\ell_t(\theta)}{\partial\theta\partial\theta^\top}$
exists and is continuous in $\Theta$. Third, it follows from Lemma 2(b)-(c) that
$\dfrac{1}{n}\sum\limits_{t=1}^n\dfrac{\partial^2\ell_t(\hat{\theta}_n)}{\partial\theta\partial\theta^\top}$ converges to $-\Sigma$ in probability.
Fourth, by Lemma 3, we have $\dfrac{1}{\sqrt n}\sum\limits_{t=1}^n\dfrac{\partial\ell_t(\theta_0)}{\partial\theta}\overset{d}{\to}N(0, \Omega)$.
Thus,  all conditions in Theorem 4.1.3 in Amemiya (1985)  hold, i.e.,
$\sqrt{n}(\hat{\theta}_n-\theta_0)\overset{d}{\to}N(0,\Sigma^{-1}\Omega\Sigma^{-1})$.  \end{proof}

\begin{proof}[Proof of Theorem 3] Note that (\ref{add1}) follows directly from (\ref{4.5}) without the involved derivations. Hence, the theorem can be shown by
repeating Lemmas 1--3  and using the same arguments in deriving (2.2) in Self and Liang (1987).
\end{proof}

{\bf Acknowledgments.~}
Ling's research was supported by the Hong Kong Research Grants Council  (Grant HKUST641912, 603413 and FSGRF12SC12).
Peng's research was supported by  NSF grant DMS-1005336 and Simons Foundation.
Zhu's research was supported by National Natural Science Foundation of China (11371168, 11271155),
Specialized Research Fund for the Doctoral Program of Higher Education (20110061110003),
Science and Technology Developing Plan of Jilin Province (20130522102JH) and Scientific Research Foundation for the Returned Overseas Chinese Scholars, State Education Ministry.

\end{document}